\documentclass[12pt]{amsart}
\usepackage[myheadings]{fullpage}
\usepackage{amsmath,amsthm}

\numberwithin{equation}{section}

\newtheorem*{theorem*}{Theorem}

\newtheorem{theorem}{Theorem}
\newtheorem{lemma}{Lemma}[section]
\newtheorem{prop}[lemma]{Proposition}
\newtheorem{cor}[lemma]{Corollary}

\theoremstyle{definition}
\newtheorem{definition}[lemma]{Definition}
\newtheorem{example}[lemma]{Example}

\newtheorem{remark}[lemma]{Remark}

\DeclareMathOperator{\area}{area}
\DeclareMathOperator{\vol}{vol}
\DeclareMathOperator{\conv}{conv}

\newcommand{\R}{\mathbb R}
\newcommand{\RP}{\mathbb{RP}}
\renewcommand{\phi}{\varphi}
\newcommand{\pd}{\partial}
\newcommand{\ep}{\varepsilon}
\newcommand{\syspi}{\text{\rm sys}\pi}

\newcommand{\vect}{\overrightarrow}
\newcommand{\rvect}{\overleftarrow}

\renewcommand{\L}{\mathcal L}

\def\be{\begin{equation}}
\def\ee{\end{equation}}

\begin{document}

\title{Filling minimality of Finslerian 2-discs}

\author{Sergei Ivanov}
\thanks{Supported by the Dynasty foundation and RFBR grant 09-01-12130-ofi-m}
\email{svivanov@pdmi.ras.ru}
\address{Saint Petersburg department of Steklov Math Institute,
191023, Fontanka 27, Saint Petersburg, Russia}

\begin{abstract}
We prove that every Riemannian metric on the 2-disc such that
all its
geodesics are minimal, is a minimal filling of its boundary
(within the class of fillings homeomorphic to the disc).
This improves an earlier result of the author
by removing the assumption that the boundary is convex.
More generally, we prove this result for
Finsler metrics with area defined as the two-dimensional
Holmes--Thompson volume.
This implies a generalization
of Pu's isosystolic inequality to Finsler metrics, both for
Holmes--Thompson and Busemann definitions of Finsler area.
\end{abstract}

\maketitle

\section{Introduction}

For a Riemannian metric $g$ on a compact manifold $M$
with boundary, let $d_g$ denote
the corresponding distance function on $M\times M$.
The \textit{boundary distance function} of $g$,
denoted by $bd_g$, is the restriction of $d_g$
to $\pd M\times\pd M$.
That is, $bd_g(x,y)$ is the length of
a $g$-shortest path in $M$ between boundary points
$x$ and~$y$.

It is natural to ask what kind of information
about $g$ can be recovered if one knows the
boundary distance function (or an approximation of it).
In some cases $bd_g$
determines $g$ uniquely up to an isometry
but in general this is not the case.
Attaching a large ``bubble''
with a narrow neck has very little effect (in the $C^0$ sense)
on boundary distances,
and changing the metric within the bubble has
no effect at all.
Thus a metric with a given boundary distance function
can be arbitrary large in terms of Riemannian volume.

However it cannot be arbitrarily small.
The boundary distance function
(or a lower bound for it) imposes a positive
lower bound on the volume of the metric.
Metrics realizing this lower bound
are called \textit{minimal fillings}, see below.
Although minimal fillings are in a sense similar
to minimal surfaces (cf.~\cite{I08}),
they do not have similar existence and regularity properties.
Nevertheless there are many examples of smooth minimal fillings
(including all metrics sufficiently
close to a flat one, cf.~\cite{BI-annals}).

It is plausible that
every smooth Riemannian metric with minimal geodesics
(see below for a precise definition) is a minimal filling.
In \cite{I02}, this conjecture was proved in dimension~2
for discs with convex boundaries.
In this paper we remove the convex boundary assumption
and generalize the result to Finsler metrics.
(The Finslerian case appears in \cite{I02}
as well but the proof there is too sketchy.)
The Finslerian result implies that
Pu's isosystolic inequality \cite{Pu} holds for
Finsler metrics.
Even in the Riemannian case the resulting proof
of Pu's inequality differs from the original one,
in particular, it does not use the uniformization theorem.

\subsection{Riemannian minimal fillings}
Let $D$ denote the two-dimensional disc $D^2$
and $S=\pd D\simeq S^1$.
For a nonnegative function $d:S\times S\to\R$,
define the \textit{filling volume}
(or \textit{filling area}) of $(S,d)$,
denoted by $\operatorname{FillVol}_D(S,d)$,
by
\be\label{fillvol-def}
 \operatorname{FillVol}_D(S,d) = \inf_{g:bd_g\ge d} \area(D,g)
\ee
where the infimum is taken over all Riemannian metrics $g$ on $D$
such that the boundary distance function $bd_g$ is bounded below by~$d$.
Here $\area(D,g)$ denotes the two-dimensional Riemannian volume
of $D$ with respect to~$g$.

The notion of filling volume was introduced by Gromov~\cite{frm}.
The above definition differs from Gromov's in one essential detail
(indicated by the subscript $D$): we restrict ourselves to metrics
on the disc while in Gromov's definition one takes the infimum
over Riemannian manifolds of varying topology
(namely all orientable manifolds) whose boundaries
are identified with~$S$.
In higher dimensions restricting the topology type does not change the
filling volume \cite[App.~2, Prop.~A$'$]{frm},
but in dimension~2 the two definitions are probably not equivalent.

A metric $g_0$ (or a space $(D,g_0)$)
is said to be a \textit{minimal filling} if
its realizes the infimum in \eqref{fillvol-def} for some $d$,
or, equivalently, for $d=bd_{g_0}$.
Substituting the definitions yields the following reformulation:
$g_0$ is a minimal filling if and only if
for every Riemannian metric $g$ on $D$ satisfying
\be\label{riem-assum}
 d_g(x,y) \ge d_{g_0}(x,y)\qquad\text{for all $x,y\in\pd D$},
\ee
one has 
\be\label{riem-conclu}
\area(D,g)\ge\area(D,g_0) .
\ee

Classic examples of two-dimensional
minimal fillings are the standard hemisphere
(this follows from Pu's inequality, see below)
and regions in the Euclidean and hyperbolic planes,
cf.~\cite{frm}. In fact, Euclidean and hyperbolic regions
are minimal in a stronger sense of Gromov's definition
while the hemisphere is known to be minimal only
within the class of fillings
homeomorphic to the disc.

We say that $g$ is a \textit{metric with minimal geodesics}
if every $g$-geodesic in the interior of $D$ is a shortest
path between its endpoints.
By continuity, minimality of all geodesics in the interior
implies minimality of geodesics with endpoints
on the boundary but otherwise contained in the interior.
We do not consider geodesics that
have points on $\pd D$ other than endpoints.

One of the goals of this paper is to prove
the following theorem.

\begin{theorem}\label{riem}
Every Riemannian metric with minimal geodesics
on $D$ is a minimal filling is the above sense.
\end{theorem}

This theorem was proved in \cite{I02} 
under an additional
assumption that the boundary of the disc is convex
with respect to the metric.
The proof in this paper is essentially the same
as in \cite{I02} modulo technical details
allowing the proof to work in the case of a non-convex boundary.
The key idea to estimate areas using cyclic order
of gradients of distance functions is borrowed
from \cite{BI02}, and it was used earlier
by Croke and Kleiner~\cite{CK95}.

\subsection{Finslerian case}
\label{intro-finsler}

Theorem \ref{riem} is a partial case of Theorem \ref{main}
which asserts the same fact for Finsler metrics
(including non-reversible ones).
We do not use heavy machinery of Finsler geometry,
all necessary definitions and facts are included
here and in section~\ref{sec-finsler}.
Details and proofs can be found e.g.\ in~\cite{Shen}.

A \textit{Finsler metric} on a smooth manifold $M$
is a continuous function $\phi:TM\to\R$ satisfying the following
conditions:

(1) $\phi(tv)=t\phi(v)$ for all $v\in TM$ and $t\ge 0$;

(2) $\phi$ is positive on $TM\setminus 0$;

(3) $\phi$ is smooth on $TM\setminus 0$
(for our purposes, $C^2$ smoothness is sufficient);

(4) $\phi$ is strictly convex in the following sense:
for every $x\in M$, the function $\phi^2|_{T_xM}$
has positive definite second derivatives on $T_xM\setminus\{0\}$.

A Finsler metric $\phi$ is \textit{reversible}
(or \textit{symmetric}) if
$\phi(-v)=\phi(v)$ for all $v\in TM$.

A Finsler metric $\phi$ on $M$ can be thought of as a family
of (non-symmetric) norms $\phi_x:=\phi|_{T_xM}$, $x\in M$,
on the fibers of $TM$.
Riemannian metrics are partial case of (reversible)
Finsler metrics, they are characterized by the property
that all norms $\phi_x$ are Euclidean.
For a Finsler metric $\phi$, one naturally defines geodesics, lengths
and a (non-symmetric) distance function
$d_\phi:M\times M\to\R_+$, cf.\ section \ref{sec-finsler}.
We define the boundary distance function and
metrics with minimal geodesics
in the same way as in the Riemannian case.

The definition of area is a more delicate subject.
There are several non-equivalent definitions of area
and volume in Finsler geometry, cf.\ \cite{Thompson}
or \cite{AT} for a survey.
The most widely used definition
is Busemann's \cite{Busemann}
where Finsler volume is defined so that the volume of a unit
ball in every $n$-dimensional normed space is the same as
that of the standard Euclidean ball in~$\R^n$.
The Busemann volume of a reversible Finsler metric equals
the Hausdorff measure of the corresponding metric space.
We denote the Busemann volume of an $n$-dimensional Finsler manifold
$(M,\phi)$ by $\vol_n^b(M,\phi)$.

In this paper we mainly use another definition,
namely the Holmes--Thompson volume \cite{HT}.
By definition, the Holmes--Thompson volume
$\vol_n^{ht}(M,\phi)$ of an $n$-dimensional
Finsler manifold $(M,\phi)$
equals the canonical (symplectic) volume of
the bundle of unit balls in $T^*M$, normalized
by a suitable constant (namely divided by
the Euclidean volume of the unit ball in $\R^n$).
See section \ref{sec-finsler} for more details.

This choice of volume definition
is enforced by the following fact: Finsler metrics
(including metrics with minimal geodesics)
admit non-isometric deformations preserving the
boundary distances. These deformations preserve
the Holmes--Thompson volume but not the Busemann
(or any other) volume.

The main result of this paper is the following theorem.

\begin{theorem}\label{main}
Let $\phi_0$ be a Finsler metric
with minimal geodesics on $D$,
and let $\phi$ be a Finsler metric on $D$
such that
$$
 d_\phi(x,y) \ge d_{\phi_0}(x,y)\qquad\text{for all $x,y\in\pd D$}.
$$
Then
$$
 \vol_2^{ht}(D,\phi) \ge \vol_2^{ht}(D,\phi_0)
$$
with equality if and only if $\phi$ is a metric
with minimal geodesic whose boundary distance function
equals that of $\phi_0$.
\end{theorem}

The proof of Theorem \ref{main} is contained in sections
\ref{sec-finsler}, \ref{sec-cyclic} and~\ref{sec-mainproof}.
Section \ref{sec-finsler} contains preliminaries
and some technical facts about Finsler metrics,
in section \ref{sec-cyclic} we obtain a lower bound for
$\vol_2^{ht}(D,\phi)$ in terms of a ``cyclic map''
(Proposition \ref{I<vol}),
and in section \ref{sec-mainproof} we construct a special
cyclic map such that the resulting lower bound equals
$\vol_2^{ht}(D,\phi_0)$.
The equality case of the theorem follows from
the inequality, this is explained in the end
of section~\ref{sec-mainproof}.
Section \ref{sec-pu} is a digression where
we prove Pu's inequality for Finsler metrics, see below.

\subsection{Pu's inequality}

For a metric $\phi$ on a manifold $M$,
denote by $\syspi_1(M,\phi)$
the one-dimensional homotopy systole,
that is the length of the shortest noncontractible
loop in $(M,\phi)$.
Certain types of manifolds (cf.\ \cite{frm} for details)
admit an inequality of the form
$$
\syspi_1(M,g)\le C(M)\cdot\vol_n(M,g)^{1/n}
$$
for any Riemannian metric $g$ on $M$,
where $n$ is the dimension and $C(M)$ 
depends only on the topology of~$M$.
The optimal value of the constant $C(M)$
is known only in a few cases, one of which
is Pu's theorem for $M=\RP^2$:

\begin{theorem*}[P.~Pu \cite{Pu}]
For every Riemannian metric $g$ on $\RP^2$,
one has
$$
 \area(\RP^2,g) \ge \frac2\pi\,\syspi_1(\RP^2,g)^2
$$
with equality if and only if $g$ has constant curvature
(or, equivalently, is isometric to a rescaling of
a standard ``round'' metric).
\end{theorem*}

It is easy to see that Pu's inequality
is equivalent to the fact that the
hemisphere $S^2_+$ with its standard Riemannian
metric is a minimal filling.
To reduce Pu's inequality to the filling minimality
of the hemisphere, 
rescale the metric so that $\syspi_1(\RP^2,g)=\pi$
and cut $\RP^2$ along a $g$-shortest
noncontractible loop. The resulting space
is a disc with a Riemannian metric such that
the length of the boundary is $2\pi$
and the distance between every pair of
opposite points of the boundary is~$\pi$.
Then the triangle inequality implies that
the distance between any two boundary points
is realized by an arc of the boundary.
This means that \eqref{riem-assum} is
satisfied if $g_0$ is the metric of the
standard hemisphere (whose boundary circle
is identified with the boundary of our disc
in a length-preserving way).

Since the geodesics in the hemisphere are minimal,
Theorem \ref{riem} implies that the hemisphere is a minimal filling.
Therefore the area of our metric on the disc
(and hence of the original metric on $\RP^2$)
is at least that of the hemisphere, that is, $2\pi$.
Thus Pu's inequality follows from Theorem \ref{riem}.

The above argument applies without changes to Finsler metrics
(except that the triangle inequality part requires
symmetry of the metric). Thus Theorem \ref{main}
implies the following Finslerian generalization
of Pu's inequality.

\begin{theorem}\label{pu}
Let $\phi$ be a reversible Finsler metric on $\RP^2$.
Then
$$
 \vol_2^{ht}(\RP^2,\phi) \ge \frac2\pi\, \syspi_1(\RP^2,\phi)^2 .
$$
\end{theorem}

Although Theorem \ref{pu} follows from Theorem \ref{main},
we prove it directly in section~\ref{sec-pu}. The reason
is that in this case the proof is simpler as it avoids
a complicated construction of section~\ref{sec-mainproof}.

It is well-known (cf.\ e.g.\ \cite{Duran} or \cite{AT})
that  the Busemann volume of any Finsler metric is
no less than its Holmes--Thompson volume,
and they are equal if and only if
the metric is Riemannian.
This fact and Theorem \ref{pu} immediately
imply the following Pu's inequality
for the Busemann area.

\begin{theorem}
Let $\phi$ be a reversible Finsler metric on $\RP^2$.
Then
$$
 \vol_2^b(\RP^2,\phi) \ge \frac2\pi\,\syspi_1(M,\phi)^2
$$
with equality if and only if
$\phi$ is a Riemannian metric of constant curvature.
\end{theorem}

\subsection{Remarks and open questions}

1. It remains unclear whether Riemannian metrics with minimal
geodesics on the disc are minimal fillings in a stronger sense,
that is, within the class of Riemannian metrics on surfaces
of arbitrary genus. This is not known even in the case
of the hemisphere although the question dates back to
Gromov's paper \cite{frm}. The proof in this paper
does not work for surfaces of higher genus since it
uses the Jordan Curve Theorem in several places
(most importantly, in Lemma \ref{cyclic-dir}
and Lemma \ref{cyclic-aux}).

2. One of the motivating reasons to study minimal fillings
is their relation to boundary rigidity problems.
A metric $g$ is said to be boundary (distance) rigid
if its boundary distance function determines the metric
uniquely up to an isometry. It is conjectured
(cf.\ e.g.\ \cite{Michel} and \cite{Croke91})
that all metrics with strongly minimal geodesics
(that is, minimal and having no conjugate points up to and including
endpoints at the boundary) are boundary rigid.
In dimension 2 this conjecture was proved
(for metrics with convex boundaries)
by Pestov and Uhlmann \cite{pestov-uhlmann}.
A promising approach to boundary rigidity is
studying the case of equality in the filling
inequality \eqref{riem-conclu},
cf.\ \cite{BI-annals} for a successful application of this approach.
Unfortunately the proof of Theorem \ref{riem} in this paper
does not suggest a way to study the equality case
because of the Finslerian nature of the proof and
non-rigidity of Finsler metrics.

3. One easily sees that the assumption about $\phi_0$
in Theorem \ref{main} is the
weakest possible: a metric with a non-minimal geodesic
cannot be a minimal filling in the Finsler category
since it can be altered so as to reduce the volume
while preserving the boundary distances
(similarly to the proof of the equality case
in section~\ref{sec-mainproof}).
This argument does not work in the Riemannian category,
and this raises the following
question: are there Riemannian minimal fillings
with non-minimal geodesics?
Probably the simplest example to study is
the product metric on $S^1\times[0,1]$.

\section{Finsler metrics, geodesics and directions}
\label{sec-finsler}


Let $\phi$ be a Finsler metric on $D$.
We omit dependence on $\phi$ in most terms
and notations.
For $x\in D$, we denote by $B_x$ and $U_x$
the unit ball and the unit sphere of
the norm $\phi_x=\phi|_{T_xD}$, that is,
$$
\begin{aligned}
 B_x &= \{v \in T_xD : \phi(v) \le 1 \}, \\
 U_x &= \{v \in T_xD : \phi(v) = 1 \}.
\end{aligned}
$$
Note that $B_x$ is a convex set whose
boundary $U_x$ is a smooth strictly convex curve.

The length $L_\phi(\gamma)$ of a piecewise smooth
curve $\gamma:[a,b]\to D$ is defined by
$$
 L_\phi(\gamma) = \int_a^b \phi(\dot\gamma(t))\,dt
$$
where the velocity vector $\dot\gamma(t)$ is regarded as
an element of $T_{\gamma(t)}D$. The distance function
$d_\phi:D\times D\to\R_+$ is defined by
$
 d_\phi(x,y) = \inf_\gamma L_\phi(\gamma)
$
where the infimum is taken over all piecewise smooth curves $\gamma$
starting at $x$ and ending at~$y$.
Note that $d_\phi$ is not symmetric (unless $\phi$ is reversible)
but it has other standard properties of a distance,
in particular the triangle inequality
\begin{equation}
\label{trineq}
d_\phi(x,y)+d_\phi(y,z)\ge d_\phi(x,z) .
\end{equation}
Once the distance is defined, the length functional
extends to all continuous curves in a usual way,
and a standard compactness argument shows that every
pair of points $x,y\in D$ can be connected by a shortest path,
i.e.\ is a curve from $x$ to $y$ whose length equals $d_\phi(x,y)$.
Note that a pointwise limit of shortest paths is a shortest path
due to lower semi-continuity of length.

A \textit{geodesic} is a curve which is contained in the interior
of $D$ except possibly the endpoints and is a critical point
of the energy functional
$\gamma\mapsto \int \phi^2(\gamma(t))\,dt$.
Finsler geodesics have standard properties
such as existence and uniqueness of a geodesic
with a given initial velocity and the fact that
every shortest path in the interior of $D$ is
a geodesic (and hence smooth). The only notable
difference from Riemannian geodesics is that
reversing direction may turn a geodesic
into a non-geodesic.
Unless otherwise stated, all geodesics and
shortest paths are assumed parameterized by arc length.

As usual, $T^*D$ denotes the co-tangent bundle of $D$;
an element of $T^*_xD\subset T^*D$
is a linear function on $T_xD$.
The dual metric $\phi^*:T^*D\to\R_+$ is defined by
$$
 \phi^*(u) = \sup\{u(v): v\in U_x\} \qquad\text{for $u\in T^*_xD$, $x\in M$} .
$$
The definition implies that $\phi^*|_{T_x^*D}$ is a
(possibly non-symmetric) norm on $T_x^*D$.
We denote by $B^*_x$ and $U^*_x$ the unit ball and the unit
sphere of this norm, that is,
$$
\begin{aligned}
 B^*_x &= \{u \in T^*_xD : \phi^*(u) \le 1 \}, \\
 U^*_x &= \{u \in T^*_xD : \phi^*(u) = 1 \}.
\end{aligned}
$$
Note that $B_x$ and $B_x^*$ are polar to each other.

Recall that the co-tangent bundle $T^*D$ carries
a canonical (four-dimensional) volume form.
By definition, the Holmes--Thompson area $\vol_2^{ht}(D,\phi)$
equals the canonical volume of the set
$B^*D=\bigcup_{x\in D} B^*$, divided by~$\pi$.
In coordinates $(x_1,x_2)$ on $D$, this
can be written as
\begin{equation}
\label{ht}
 \vol_2^{ht}(D,\phi) = \frac1\pi\int_D |B^*_x|\,dx_1dx_2
\end{equation}
where $|B^*_x|$ is the coordinate Lebesgue measure
of the set $B^*_x\subset T^*_xD\simeq\R^2$.
In the sequel we always use this coordinate
formula rather than any of the invariant expressions
for the Holmes--Thompson area.

The \textit{Legendre transform} associated with $\phi$
is a norm-preserving positively homogeneous map
$\L=\L_\phi:TD\to T^*D$ that can be defined
as follows: for every $x\in M$ and $v\in U_x$,
$\L(v)$ is the unique co-vector $u\in U^*_x$
such that $u(v)=1$.
Strict convexity of $\phi$ implies that
$\L$ is a diffeomorphism between $U_x$ and $U^*_x$.

\begin{definition}
Let $x\in D\setminus\pd D$ and $y\in D$.
Consider a shortest path $\gamma:[0,T]\to D$
from $x$ to~$y$. Since $x\notin\pd D$, $\gamma$
is a geodesic near $x$ and hence is differentiable at~0.
We refer to the initial velocity vector $\dot\gamma(0)$
of $\gamma$
as a \textit{direction to $y$ at $x$} and denote
it by $\vect{xy}$.
If all shortest paths from $x$ to $y$ have the
same initial velocity, we say that
$\vect{xy}$ is \textit{uniquely defined}.

Similarly, if $\gamma:[0,T]\to D$ is a shortest
path from $y$ to $x$, we refer to the vector
$\dot\gamma(T)$ as the \textit{direction from $y$ at $x$}
and denote it by $\rvect{xy}$;
and we say that $\rvect{xy}$ is uniquely defined
if this vector is the same for all such paths.
\end{definition}

\begin{lemma}\label{min-unique}
Suppose that $\phi$ is a metric with minimal geodesics.
Then for every pair of distinct points
$x\in D\setminus\pd D$ and $y\in D$,
the directions $\vect{xy}$ and $\rvect{xy}$
are uniquely defined.
\end{lemma}

\begin{proof}
We prove uniqueness of $\vect{xy}$,
the case of $\rvect{xy}$ is similar.
Suppose the contrary, then there exist
shortest paths $\gamma_1,\gamma_2:[0,T]\to D$
connecting $x$ to~$y$ such that
$\dot\gamma_1(0)\ne\dot\gamma_2(0)$.
Let $z$ be the nearest to~$x$
common point of $\gamma_1$ and $\gamma_2$.
Then, by Jordan Curve Theorem, the intervals of
$\gamma_1$ and $\gamma_2$ between $x$ and $z$
bound a region $\Omega\subset D$ containing
no points of $\pd D$.
Let $\gamma:[0,T_1]\to D$ be a geodesic starting at $x$
with initial velocity $\dot\gamma(0)$ pointing
into $\Omega$ and extended forward until it reaches the boundary
of $\Omega$
(due to minimality of geodesics, $\gamma$ cannot have infinite length
and hence hits the boundary eventually).
Denote $p=\gamma(T_1)$ and assume for definiteness
that $p$ lies on~$\gamma_1$.
Since $\phi$ is a metric with minimal geodesics,
$\gamma$ is a shortest path.

Let $s$ denote the interval of $\gamma_1$ from $x$ to~$p$.
Then $s$ and $\gamma$ are shortest paths,
hence $L_\phi(s)=L_\phi(\gamma)$.
Let $\gamma^-:[-\ep,0]\to D$ be a geodesic
extending $\gamma$ backwards
(that is, $\gamma^-(0)=x$ and $\dot\gamma^-(0)=\dot\gamma(0)$).
Then $\gamma^-\cup\gamma$ is a geodesic and hence a shortest path.
Therefore the curve $\gamma^-\cup s$ is a shortest path
because it has the same length. But this curve is not smooth at~$0$
since $\dot s(0)=\dot\gamma_1(0)\ne \dot\gamma(0)$,
a contradiction.
\end{proof}

Fix an orientation of $D$.
This orientation induces orientations
and hence cyclic orders on $\pd D$ and the circles
$U_x$ and $U^*_x$ for all $x\in D$.
Note that the Legendre transform
${\L:U_x\to U^*_x}$ is an orientation-preserving
diffeomorphism.

\begin{lemma}\label{cyclic-dir}
Let $x\in D\setminus\pd D$ and $p_1,p_2,p_3\in\pd D$.
Suppose that directions
$\vect{xp_1}$, $\vect{xp_2}$ and $\vect{xp_3}$
are uniquely defined and distinct.
Then the cyclic ordering of the vectors
$\vect{xp_1}$, $\vect{xp_2}$ and $\vect{xp_3}$
in $U_x$ is the same as the cyclic ordering
of the points $p_1$, $p_2$ and $p_3$ in $\pd D$.

The same is true for
$\rvect{xp_1}$, $\rvect{xp_2}$ and $\rvect{xp_3}$,
provided that they are uniquely defined and distinct.
\end{lemma}

\begin{proof}
We will prove the first statement, the second one is similar.
For every $i=1,2,3$, let $\gamma_i:[0,T_i]\to D$ be a shortest path
from $x$ to $p_i$ to $x$. Then $\dot\gamma_i(0)=\vect{xp_i}$.

We claim that the curves $\gamma_1$, $\gamma_2$ and $\gamma_3$ have no common
points except $x$. Indeed, suppose that $\gamma_1$ and $\gamma_2$
have a common point $q\ne x$.
Let $s_1$ and $s_2$ denote the intervals of $\gamma_1$ and $\gamma_2$
between $x$ and~$q$. Since $s_1$ and $s_2$ are intervals of shortest paths,
they are shortest paths themselves. In particular,
$L_\phi(s_1)=L_\phi(s_2)=d_\phi(x,q)$.
Consider a new curve $\gamma$ composed from $s_2$ and the segment of $\gamma_1$
between $q$ and $p_1$. Since $L_\phi(s_1)=L_\phi(s_2)$,
we have $L_\phi(\gamma)=L_\phi(\gamma_1)$, hence $\gamma$
is another shortest path from $x$ to $p_1$.
However the initial velocity of $\gamma$
equals $\dot\gamma_2(0)=\vect{xp_2}\ne\vect{xp_1}$, contrary
to the assumption that $\vect{xp_1}$
is uniquely defined. The claim follows.

Since the curves $\gamma_1$, $\gamma_2$ and $\gamma_3$
have no common points except $x$ and our space $D$
is a topological 2-disc, the cyclic ordering
of the points $p_1$, $p_2$ and $p_3$ on $\pd D$
is the same as that of the intersections
$\gamma_1$, $\gamma_2$ and $\gamma_3$ with a small circle
centered at~$x$, and the latter is the same as the
cyclic ordering of the initial velocity vectors
$\dot\gamma_1(0)$, $\dot\gamma_2(0)$ and $\dot\gamma_3(0)$.
\end{proof}

\begin{definition}\label{fwlip}
A function $f:D\to\R$ is said to be \textit{forward 1-Lipschitz}
(with respect to $\phi$) if
$
 f(y) - f(x) \le d_\phi(x,y)
$
for all $x,y\in D$.
\end{definition}

\begin{example}\label{lip-dist}
For any $p\in D$, the functions $x\mapsto d_\phi(p,x)$
and $x\mapsto -d_\phi(x,p)$ are forward 1-Lipschitz.
The requirement of Definition \ref{fwlip} follows from
the triangle inequality \eqref{trineq}.
\end{example}

Obviously every forward 1-Lipschitz function is Lipschitz
in any local coordinates.
Hence, by Rademacher's theorem, such a function is differentiable almost
everywhere. We denote the derivative of $f$ at $x\in D$
by $d_xf$ and the map $x\mapsto d_xf\in T^*_xD$ by $df$.
We regard $df$ as a differential 1-form on $D$ with
Borel measurable coefficients (defined a.~e.).

Note that $\phi^*(d_xf)\le 1$ if $f$ is a forward 1-Lipschitz
function differentiable at~$x$.

\begin{lemma}\label{dist-grad}
Let $f:D\to\R$ be a forward 1-Lipschitz function
and $x\in D\setminus\pd D$.
Suppose that $f$ is differentiable at~$x$. Then

1. If a point $y\in D$ satisfies
\begin{equation}
\label{dl1}
 f(y) = f(x) + d_\phi(x,y) ,
\end{equation}
then $\vect{xy}$ is uniquely defined
and $d_xf = \L(\vect{xy})\in U^*_x$.

2. If a point $y\in D$ satisfies
\begin{equation}
\label{dl2}
 f(x) = f(y) + d_\phi(y,x) ,
\end{equation}
then $\rvect{xy}$ is uniquely defined
and $d_xf = \L(\rvect{xy})\in U^*_x$.
\end{lemma}

\begin{proof}
Let $y\in D$ satisfy \eqref{dl1}
and let $v$ be a direction to $y$ at~$x$.
Then $v=\dot\gamma(0)$ where
$\gamma:[0,T]\to D$ is a shortest path
from $x$ to $y$ parameterized by arc length.
Since $f$ is forward 1-Lipschitz, the function
$t\mapsto f(\gamma(t))$ is 1-Lipschitz on $[0,T]$.
On the other hand,
$$
f(\gamma(T))-f(\gamma(0))=f(y)-f(x)=d_\phi(x,y)=T ,
$$
therefore 
$f(\gamma(t))=f(x)+t$ for all $t\in[0,T]$.
Hence
$$
 d_xf(v) = \frac d{dt}\Big|_{t=0} f(\gamma(t)) = 1 .
$$
Since $\phi^*(d_xf)\le 1$ and $\phi(v)=1$,
this identity implies that $\phi^*(d_xf)=1$
and $d_xf = \L(v)$.
This determines $v$ uniquely
(namely $v=\L^{-1}(d_xf)$)
and the first assertion of the lemma follows.

The second assertion follows by a similar argument
applied to a shortest path from $y$ to $x$ and
its derivative at the endpoint.
\end{proof}

\begin{cor}\label{dist-grad-cor}
Fix a point $p\in D$ and consider a function
$f:D\to\R$ given by $f(x)=d_\phi(p,x)$.
Then, for every $x\in D\setminus\pd D$ where
$f$ is differentiable, the direction $\rvect{xp}$
is uniquely defined and $d_xf = \L(\rvect{xp})\in U^*_x$.
\end{cor}

\begin{proof}
The function $f$ is forward 1-Lipschitz,
cf.\ Example \ref{lip-dist}.
For every $x\in D$, the point $y=p$
satisfies \eqref{dl2} from Lemma \ref{dist-grad},
hence the result.
\end{proof}

\section{Cyclic maps}
\label{sec-cyclic}

\begin{definition}\label{def-cyclic}
A Lipschitz map $f:D\to\R^n$, $f=(f_1,f_2,\dots,f_n)$ is said to be \textit{cyclic}
(with respect to $\phi$) if the derivatives of its coordinate
functions $f_i$ satisfy the following:

(1) if $x\in D\setminus\pd D$ and $f_i$ is differentiable at $x$,
then $d_xf_i\in U^*_x$;

(2) for every $x\in D\setminus\pd D$ and $i,j,k\in\{1,\dots,n\}$
such that $i<j<k$ and
the derivatives $d_xf_i$, $d_xf_j$, $d_xf_k$ are
well-defined and distinct,
the cyclic ordering of the triple $(d_xf_i, d_xf_j, d_xf_k)$
in~$U^*_x$ is positive.
\end{definition}

Note that the second requirement depends only on
the cyclic order of coordinate functions. Thus
if a map $(f_1,f_2,\dots,f_n)$ is cyclic,
then so is $(f_2,f_3,\dots,f_n,f_1)$.

\begin{example}\label{cyclic-dist}
Let $p_1,p_2,\dots,p_n$ a be cyclically ordered collection
of points in $\pd D$.
Define $f_i(x) = d_\phi(p_i,x)$ for all $x\in D$, $i\le n$.
Then the map $f=(f_1,f_2,\dots,f_n)$ is cyclic.
\end{example}

\begin{proof}
If $f_i$ is differentiable at $x\in D\setminus\pd D$,
then $d_xf_i=\L(\rvect{xp_i})\in B^*_x$
by Corollary \ref{dist-grad-cor}.
Since $\L$ is an orientation-preserving
diffeomorphism between $U_x$ and $U^*_x$,
the cyclic ordering of any triple
$(d_xf_i,d_xf_j,d_xf_k)$ in $U^*_x$
is the same as that of the triple
$(\rvect{xp_i},\rvect{xp_j},\rvect{xp_k})$ in $U_x$.
By Lemma \ref{cyclic-dir}, the latter is
the same as the cyclic ordering of $(p_i,p_j,p_k)$
in~$\pd D$, provided that these (co-)vectors are
well-defined and distinct.
And the cyclic ordering of $(p_i,p_j,p_k)$
is positive if $i<j<k$.
\end{proof}

\begin{definition}
For a Lipschitz map $f=(f_1,\dots,f_n):D\to\R^n$,
define
\begin{equation}
\label{I}
I(f) = \frac1{2\pi}\int_D \sum_{i=1}^n df_i \wedge df_{i+1} .
\end{equation}
Here and in the sequel all indices like $i+1$ are taken modulo $n$.
\end{definition}

\begin{lemma}\label{l:I-by-boundary}
$I(f)$ is determined by the
restriction of $f$ to the boundary.
That is, if $f$ and $\overline f$ are Lipschitz
maps from $D$ to $\R^n$ and $f|_{\pd D}=\overline f|_{\pd D}$,
then $I(f)=I(\overline f)$.
\end{lemma}

\begin{proof}
The (Borel measurable) 2-form $\sum df_i \wedge df_{i+1}$
on $D$ is induced by $f$
from a 2-form $\omega=\sum dx_i\wedge dx_{i+1}$ on $\R^n$.
Hence the integral in \eqref{I} is the integral of $\omega$
over a Lipschitz singular chain defined by $f$. This integral
is determined by $f|_{\pd D}$ since $\omega$ is closed.
\end{proof}

\begin{remark}\label{r:I-by-boundary}
Since the 2-form $\omega$ in the above proof
is the exterior derivative
of a 1-form $\sum x_i\, dx_{i+1}$, one can explicitly
rewrite $I(f)$ using Stokes' formula:
$$
 I(f) = \frac1{2\pi}\int_{\pd D} \sum_{i=1}^n f_i \cdot df_{i+1} .
$$
Each term $\int_D df_i \wedge df_{i+1} = \int_{\pd D} f_i \cdot df_{i+1}$
is the oriented area (with multiplicities) encircled by the loop
$F_i|_{\pd D}$ in the plane where $F_i$ is a map from $D$ to $\R^2$
defined by $F_i(x)=(f_i(x),f_{i+1}(x))$.
\end{remark}

Fix a coordinate system $(x_1,x_2)$ in $D$.
This induces coordinates in the co-tangent bundle $T^*D$,
that is, every fiber $T^*_xD$ is identified with $\R^2$.
For a measurable set $E\subset T_x^*D$
we denote by $|E|$ its Lebesgue measure
w.r.t.\ these coordinates.

For a set $A\subset T^*_xD$, let $\conv(A)$ denote
its convex hull (that is, the least
convex set containing $A$). If $A$ is finite,
then $\conv(A)$ is either a convex polygon
in the plane $T^*_xD\simeq\R^2$
or a line segment or a single point.

\begin{lemma}\label{l-conv}
For a cyclic map $f=(f_1,\dots,f_n):D\to\R^n$ one has
$$
I(f) = \frac1\pi \int_D |\conv\{d_xf_1,\dots,d_xf_n\}|\,dx .
$$
where $dx=dx_1dx_2$ denotes the coordinate integration.
\end{lemma}

\begin{proof}
Let $S_i$ be a Borel measurable function on $D$
defined by the relation
$$
 df_i\wedge df_{i+1} = 2S_i \,dx_1\wedge dx_2 .
$$
Then
$$
 I(f) = \frac1\pi\int_D \sum_{i=1}^n S_i(x) \,dx ,
$$
and it suffices to prove that
\begin{equation}
\label{conv1}
 \sum_{i=1}^n S_i(x) = |\conv\{d_xf_1,\dots,d_xf_n\}|
\end{equation}
for a.e.\ $x\in D$.
We will show that \eqref{conv1} holds for every $x\in D\setminus\pd D$
where $f$ is differentiable.
Fix such a point $x$ and denote $w_i=d_xf_i$
for all $i=1,\dots,n$.
Observe that $S_i(x)=\frac12\cdot\frac{w_1\wedge w_2}{dx_1\wedge dx_2}$
equals the oriented area of the Euclidean triangle
$\Delta_i:=\triangle 0w_iw_{i+1}$ in the plane $T_x^*D\simeq\R^2$.
Hence the left-hand side of \eqref{conv1} equals the sum
of oriented areas of these triangles.

Since $f$ is cyclic, we have $\phi_x^*(w_i)=1$
(cf.\ Definition \ref{def-cyclic}(1)).
Hence the points $w_1,\dots,w_n$ belong to
the convex curve $U^*_x=\pd B^*_x$.
The second requirement of Definition \ref{def-cyclic}
implies that
$w_i$, $w_j$ and $w_k$ are positively ordered
whenever they are distinct and $i<j<k$.

If all the points $w_i$ coincide, then all terms in \eqref{conv1}
are zero. Otherwise we may assume that $w_n\ne w_1$ and
furthermore that $w_{i+1}\ne w_i$ for all~$i$.
Indeed, if $w_{i+1}=w_i$, we can remove $w_{i+1}$ from the list;
this clearly does not change the right- and left-hand side of \eqref{conv1}.

If the points $w_1,\dots,w_n$ are distinct,
then the second requirement of Definition \ref{def-cyclic}
implies that they
are positively cyclically ordered.
That is, they are vertices of a convex polygon inscribed in $U^*_x$,
enumerated according to their cyclic order. Observe that this polygon and
the convex hull in the right-hand side of \eqref{conv1} are the same set.
Now \eqref{conv1} follows from the fact that
the area of this polygon equals the sum of oriented areas
of the triangles $\Delta_i=\triangle 0w_iw_{i+1}$.

It remains to consider the case when some of the points $w_i$
coincide. Recall that $w_{i+1}\ne w_i$ for all $i$.
We may assume that $w_1=w_k$ for some $k$, $2<k<n$.
Then for every $i>k$, the point
$w_i$ coincides with either $w_1$ or $w_2$,
otherwise the triples $(w_1,w_2,w_i)$ and
$(w_2,w_k,w_i)=(w_2,w_1,w_i)$ have opposite
cyclic orderings, contrary to
the definition of a cyclic map.
In particular, $w_{k+1}=w_2$ since $w_{k+1}\ne w_k=w_1$.
Then a similar argument shows that for all $i$ between 1 and $k$
the point $w_i$ coincides with either $w_k=w_1$ or $w_{k+1}=w_2$.

Thus every point $w_i$ coincides with either $w_1$ or $w_2$.
Since $w_{i+1}\ne w_i$ for all $i$,
the sequence $w_1,\dots,w_n$ consists of
alternating $w_1$ and $w_2$
(in particular, $n$ is even).
Then $\conv\{w_1,\dots,w_n\}$ is the line segment
$[w_1,w_2]$ and hence the right-hand side of \eqref{conv1} is zero.
The left-hand side is the sum of oriented areas of
the triangles $\Delta_i$
where $\Delta_i=\triangle 0w_1w_2$ if $i$ is odd
and $\Delta_i=\triangle 0w_2w_1$ if $i$ is even.
These two types of triangles have the same area
but opposite orientations, 
hence the total sum of their oriented areas is zero.
\end{proof}

\begin{prop}\label{I<vol}
$
I(f) \le \vol_2^{ht}(D,\phi)
$
for every cyclic map $f:D\to\R^n$.
\end{prop}

\begin{proof}
Let $x\in D\setminus\pd D$ be a point where $f$ is differentiable.
Then $\phi^*(d_xf_i)=1$ for all $i=1,\dots,n$
(cf.\ Definition \ref{def-cyclic}(1)),
hence $d_xf_i\in B^*_x$.
Since $B^*_x$ is convex, it follows that
$$
 \conv\{d_xf_1,\dots,d_xf_n\} \subset B_x,
$$
hence
$$
 |\conv\{d_xf_1,\dots,d_xf_n\}| \le |B_x| .
$$
Integrating over $D$ yields
$$
I(f) = \frac1\pi \int_D |\conv\{d_xf_1,\dots,d_xf_n\}|\,dx
\le \frac1\pi\int_D |B^*_x|\,dx =\vol_2^{ht}(D,\phi)
$$
Here the first equality follows from
Lemma \ref{l-conv} and the second one from \eqref{ht}.
\end{proof}

The next proposition is used in section \ref{sec-mainproof}
but not in section \ref{sec-pu}.

\begin{prop}\label{approx-santalo}
If $\phi$ is a metric with minimal geodesics,
then for every $\ep>0$ there exist a cyclically
ordered set of points $p_1,\dots,p_n\in\pd D$
such that the map
$f=(f_1,\dots,f_n):D\to\R^n$ where $f_i(x)=d_\phi(p_i,x)$
(cf.\ Example \ref{cyclic-dist}) satisfies
\begin{equation}
\label{eq-ap-santalo}
 I(f) > \vol_2^{ht}(D,\phi)-\ep .
\end{equation}
\end{prop}

\begin{proof}
Let $Q=\{q_i\}_{i=1}^\infty$ be a countable dense subset of $\pd D$.
For each $n$, define a map $f=f^{(n)}$ using
points $p_1,\dots,p_n$ obtained from the set $\{q_1,\dots,q_n\}$
by enumeration according to the cyclic order.
We are going to prove that $I(f^{(n)})\to \vol_2^{ht}(D,\phi)$
as $n\to\infty$; then for a sufficiently large $n$
the map $f=f^{(n)}$ satisfies \eqref{eq-ap-santalo}.

Since $f^{(n)}$ is cyclic (cf.\ Example \ref{cyclic-dist}),
Lemma \ref{l-conv} implies that
$$
 I(f^{(n)}) = \frac1\pi \int_D |\conv\{d_xg_1,\dots,d_xg_n\}|\,dx,
$$
where $g_i(x)=d_\phi(q_i,x)$.
As $n\to\infty$, the convex hull in the right-hand side
monotonically converges to the convex hull of the set $\{dg_i\}_{i=1}^\infty$.
Hence, by Levy's theorem,
$$
 \lim_{n\to\infty} I(f^{(n)}) = \frac1\pi \int_D \bigl|\conv(\{d_xg_i\}_{i=1}^\infty)\bigr|\, dx
$$
Taking into account \eqref{ht},
it suffices to prove that
\begin{equation}
\label{santalo-eq2}
 \bigl|\conv(\{d_xg_i\}_{i=1}^\infty)\bigr| = |B^*_x|
\end{equation}
for a.e.\ $x\in D$.
We will show that \eqref{santalo-eq2} holds for every
$x\in D\setminus\pd D$ where all functions $g_i$
are differentiable. Fix such a point $x$.
It suffices to prove that the set $\{d_xg_i\}_{i=1}^\infty$
is dense in $U^*_x$.

By Lemma \ref{min-unique}, for every $q\in\pd D$
the vector $\rvect{xq}$ is uniquely defined
and hence depend continuously on~$q$.
Define a map $\xi:\pd D\to U_x$ by
$\xi(q)=\rvect{xq}$. This map is surjective.
Indeed, for a vector $v\in U_x$ consider 
a geodesic $\gamma$
with initial data $\gamma(0)=x$ and $\dot\gamma(0)=v$ 
extended backwards up to the boundary.
By minimality of geodesics it indeed reaches the boundary
at some point $q=\gamma(-T)$, and then
$v=\rvect{qx}=\xi(q)$.
Since $\xi$ is surjective and $Q$ is dense in $\pd D$,
the set $\xi(Q)=\{\rvect{xq_i}\}_{i=1}^\infty$ is dense in~$U_x$.
By Corollary \ref{dist-grad-cor} we have
$d_xg_i=\L(\rvect{xq_i})$,
hence the set $\{d_xg_i\}_{i=1}^\infty$
is dense in $U^*_x$.
Therefore its convex hull contains the interior of $B^*_x$.
This implies \eqref{santalo-eq2}, and the proposition follows.
\end{proof}

\section{Proof of Pu's inequality}
\label{sec-pu}

The goal of this section is to give a direct proof
of Theorem \ref{pu}.
As explained in the introduction, it suffices
to prove the following:
if $\phi$ is a reversible Finsler metric on $D$
such that $L_\phi(\pd D)=2\pi$ and  for every
$x,y\in\pd D$ the distance $d_\phi(x,y)$ is realized
by an arc of $\pd D$,
then $\vol_2^{ht}(D,g)\ge 2\pi$.

Let $\phi$ be such a metric.
Fix a large positive integer $n$ and let
$p_1,p_2,\dots,p_n\in\pd D$ be a cyclically ordered
collection of points dividing $\pd D$ into $n$ arcs
of length $2\pi/n$.
Define a cyclic map $f:D\to\R^n$ as in Example \ref{cyclic-dist},
namely $f=(f_1,\dots,f_n)$ where $f_i(x)=d_\phi(p_i,x)$.

Our plan is to compute $I(f)$ and then use
Proposition \ref{I<vol} to estimate the area of the metric.
Recall that $I(f)$ can be recovered from the restriction
of $f$ to the boundary (cf.\ Lemma \ref{l:I-by-boundary}).
This restriction does not depend on $\phi$
since $d_\phi|_{\pd D\times\pd D}$ is just the intrinsic
distance of the boundary (which is isometric to the
standard circle of length $2\pi$).

To find $I(f)$, fix an $i\le n$ and consider
a map $F_i=(f_i,f_{i+1}):D\to\R^2$. Then the term
$$
 I_i(f) := \int_D df_i\wedge df_{i+1}
$$
equals the oriented area encircled by
the planar curve $F_i(\pd D)$, cf.\ Remark \ref{r:I-by-boundary}.
Let $s:[0,2\pi]\to\pd D$
be a positively oriented arc-length parameterization
of $\pd D$ such that $s(0)=s(2\pi)=p_i$.
Computing the distances along the circle yields that
$$
F_i(s(t)) =
\begin{cases}
 (t,\tfrac{2\pi}n-t), & t\in [0,\tfrac{2\pi}n], \\
 (t,t-\tfrac{2\pi}n), & t\in[\tfrac{2\pi}n,\pi], \\
 (2\pi-t,t-\tfrac{2\pi}n), & t\in[\pi,\pi+\tfrac{2\pi}n] , \\
 (2\pi-t,2\pi+\tfrac{2\pi}n-t), & t\in[\pi+\tfrac{2\pi}n,2\pi] . \\
\end{cases}
$$
This means that the curve $F_i(\pd D)$
bounds the planar rectangle with vertices
$(0,\tfrac{2\pi}n)$, $(\tfrac{2\pi}n,0)$,
$(\pi,\pi-\tfrac{2\pi}n)$ and $(\pi-\tfrac{2\pi}n,\pi)$
whose area equals
$2\cdot \tfrac{2\pi}n\cdot(\pi-\tfrac{2\pi}n)$.
Thus
$$
 I_i(f) = 2\cdot \tfrac{2\pi}n\cdot(\pi-\tfrac{2\pi}n)
$$
for all $i$, hence
$$
 I(f) = \frac1{2\pi} \sum_{i=1}^n I_i(f) = \frac n{2\pi}\cdot I_1(f) = 2\pi(1-\tfrac2n) .
$$
By Proposition \ref{I<vol},
$$
 \vol_2^{ht}(D,\phi) \ge I(f) = 2\pi(1-\tfrac2n) .
$$
Since $n$ is arbitrarily large, this inequality implies that
$\vol_2^{ht}(D,\phi)\ge 2\pi$, and
Theorem \ref{pu} follows.

\section{Proof of Theorem \ref{main}}
\label{sec-mainproof}

Let $\phi$ and $\phi_0$ be as in the theorem,
that is, $\phi_0$ is a metric with minimal geodesics
and $d_\phi(x,y) \ge d_{\phi_0}(x,y)$
for all $x,y\in\pd D$.

For every $p\in\pd D$, define a function
$f_p:D\to\R$ by
\begin{equation}
\label{specfun}
 f_p(x) =\sup_{q\in\pd D} \{ d_{\phi_0}(p,q) - d_\phi(x,q) \} .
\end{equation}
Observe that $f_p$ is forward 1-Lipschitz (with respect to~$\phi$).
Indeed, for every $q\in\pd D$ the function
$x\mapsto -d_\phi(x,q)$ is forward 1-Lipschitz
(cf.\ Example \ref{lip-dist}), hence so is
the function $x\mapsto d_{\phi_0}(p,q) - d_\phi(x,q)$,
and the supremum of a family of forward 1-Lipschitz functions
is forward 1-Lipschitz as well.

\begin{lemma}\label{same-on-boundary}
If $x\in\pd D$, then $f_p(x)=d_{\phi_0}(p,x)$.
\end{lemma}

\begin{proof}
Let $x\in\pd D$.
Then $d_\phi(x,q)\ge d_{\phi_0}(x,q)$ for every $q\in\pd D$
by the assumption of the theorem.
Hence
$$
 d_{\phi_0}(p,q) - d_\phi(x,q) \le d_{\phi_0}(p,q) - d_{\phi_0}(x,q)
 \le d_{\phi_0}(p,x)
$$
by the triangle inequality.
Taking the supremum over $q\in\pd D$ yields
that
$$
f_p(x)\le d_{\phi_0}(p,x).
$$
The inverse inequality follows by substituting $q=x$ in \eqref{specfun}:
$$
f_p(x) =\sup_{q\in\pd D} \{ d_{\phi_0}(p,q) - d_\phi(x,q) \}
\ge d_{\phi_0}(p,x) - d_\phi(x,x) =d_{\phi_0}(p,x).
$$
Thus $f_p(x)=d_{\phi_0}(p,x)$.
\end{proof}

\begin{prop}\label{special-is-cyclic}
Let $p_1,\dots,p_n\in\pd D$ be a cyclically ordered
collection of points.
Then a map $f:D\to\R$ defined by
$
f = (f_{p_1},\dots,f_{p_n})
$
is cyclic (with respect to $\phi$).
\end{prop}

\begin{proof}
%
Fix a point $x\in D\setminus\pd D$.
We say that a point $q_0\in\pd D$ is a
\textit{point of maximum} for $p\in\pd D$
if the supremum in \eqref{specfun}
is attained at~$q=q_0$, that is,
\begin{equation}
\label{pmax-eq}
 f_p(x) = d_{\phi_0}(p,q_0) - d_\phi(x,q_0) .
\end{equation}
By compactness, a point of maximum exists
for every $p\in\pd D$.

\begin{lemma}\label{pmax}
If $f_p$ is differentiable at $x$ and
$q_0$ is a point of maximum for $p$,
then the direction $\vect{xq_0}$
is uniquely defined and
$d_xf_p = \L(\vect{xq_0})\in U^*_x$.
\end{lemma}

\begin{proof}
By Lemma \ref{same-on-boundary} and \eqref{pmax-eq} we have
$
 f_p(q_0) = d_{\phi_0}(p,q_0)
 = f_p(x)+ d_\phi(x,q_0)
$.
Since $f_p$ is forward 1-Lipschitz,
this identity and Lemma \ref{dist-grad}(1) imply
the desired assertion.
\end{proof}

For every $i=1,\dots,n$, let $q_i\in\pd D$
be a point of maximum for $p_i$.
Then by Lemma \ref{pmax} we have $d_xf_{p_i}\in U^*_x$,
hence the first requirement of
Definition \ref{def-cyclic} 
 is satisfied.

Since the second requirement deals with triples
of coordinate functions, it suffices to verify it
for $n=3$. Suppose that the
derivatives $d_xf_i$, $i=1,2,3$, are well-defined
and distinct. Then by Lemma \ref{pmax} the directions $\vect{xq_i}$
are uniquely defined and $d_xf_{p_i}=\L(\vect{xq_i})$.
Hence the vectors $\vect{xq_i}$ are distinct and
their cyclic ordering in $U_x$ is the same as that
of the co-vectors $d_xf_i$ in $U^*_x$.
By Lemma \ref{cyclic-dir}, the former is the
same as the cyclic ordering of points $q_i$ in~$\pd D$
(note that the points $q_i$ are distinct since
the directions $\vect{xq_i}$ are distinct).
Thus it suffices to prove that the cyclic ordering
of the points $q_1,q_2,q_3$ in $\pd D$ is the
same as that of points $p_1,p_2,p_3$.

\begin{lemma}\label{cyclic-aux}
1. $q_i\ne p_i$ for every $i\in\{1,2,3\}$.

2. If $i,j\in\{1,2,3\}$ and $i\ne j$
then the set $\{p_i,q_j\}$
does not separate points $p_j$ and $q_i$ in $\pd D$,
that is, $p_j$ and $q_i$ belong to the same
connected component of $\pd D\setminus\{p_i,q_j\}$.
\end{lemma}

\begin{proof}
Suppose the contrary. Without loss of generality
we may assume that ${p_1=q_1}$ or
$\{p_1,q_2\}$ separates $\{p_2,q_1\}$ in~$\pd D$.
In either case, any curve connecting $p_1$ to $q_2$
intersects (possibly at an endpoint)
any curve connecting $p_2$ to $q_1$.
Let $z$ be a common point of a $\phi_0$-shortest path
from $p_1$ to $q_2$ and
a $\phi_0$-shortest path from $p_2$ to $q_1$.
Then
$$
\begin{aligned}
 d_{\phi_0}(p_1,q_2)+d_{\phi_0}(p_2,q_1)
 &=d_{\phi_0}(p_1,z)+d_{\phi_0}(z,q_2)+d_{\phi_0}(p_2,z)+d_{\phi_0}(z,q_1) \\
 &\ge d_{\phi_0}(p_1,q_1)+d_{\phi_0}(p_2,q_2)
\end{aligned}
$$
by the triangle inequality. Therefore
$$
\begin{aligned}
f_{p_1}(x)+f_{p_2}(x)
&\ge \big(d_{\phi_0}(p_1,q_2)-d_\phi(x,q_2)\big)+\big(d_{\phi_0}(p_2,q_1)-d_\phi(x,q_1)\big) \\
&\ge d_{\phi_0}(p_1,q_1)-d_\phi(x,q_1) + d_{\phi_0}(p_2,q_2)-d_\phi(x,q_2) \\
&= f_{p_1}(x) + f_{p_2}(x) .
\end{aligned}
$$
Here the first inequality follows from the definition \eqref{specfun}
of $f_{p_1}$ and $f_{p_2}$
and the last identity from the fact that $q_1$ and $q_2$ are points of
maximum for $p_1$ and $p_2$, resp.
Since the left- and right-hand side are the same, the intermediate
inequalities turn to equalities.
In particular, $f_{p_1}(x)=d_{\phi_0}(p_1,q_2)-d_\phi(x,q_2)$.
This means that $q_2$ is a point of maximum for $p_1$,
hence $d_xf_{p_1}=\L(\vect{xq_2})\ne\L(\vect{xq_1})$, a contradiction.
\end{proof}

Now the desired coincidence of cyclic orderings follows
from the following combinatorial lemma.

\begin{lemma}\label{cyclic-aux2}
Let $\{p_i\}_{i=1}^3$ and $\{q_i\}_{i=1}^3$ be two triples
of distinct points in $\pd D$ such that the assertion
of Lemma \ref{cyclic-aux} holds. Then the cyclic ordering
of these two triples is the same.
\end{lemma}

\begin{proof}
This fact was proved in \cite{I02} by
examination of possible configurations of six
points on a circle.
Here we give a more algebraic-style proof.

For distinct points
$a,b,c\in\pd D$, we define $[abc]=1$ if the cyclic ordering
of the triple $(a,b,c)$ is positive and $[abc]=-1$ otherwise.
This notation satisfies a trivial identity
\begin{equation}
\label{cyclic0}
 [abc]^2=1,
\end{equation}
is skew-symmetric:
\begin{equation}
\label{cyclic1}
 [abc]=[bca]=[cab]=-[bac]=-[acb]=-[cba]
\end{equation}
and satisfies the standard cyclic order identity
$
 [abc][bcd][cda][dab]=1
$
for any four distinct point $a,b,c,d\in\pd D$.
Taking into account \eqref{cyclic1}, this identity
can be rewritten as
\begin{equation}\label{cyclic2}
[abc]=[abd][acd][bcd] ,
\end{equation}
and
\begin{equation}\label{cyclic3}
[abc][acd][abd] = [bcd] .
\end{equation}
The fact that $\{p_i,q_j\}$
does not separate $p_j$ from $q_i$
means that
\begin{equation}\label{cyclic4}
[p_ip_jq_j] = [p_iq_iq_j]
\end{equation}
provided that the four points are distinct.

Observe that a small perturbation of the configuration
does not break the assumptions of the lemma
and does not change the cyclic ordering of the triples
$\{p_i\}$ and $\{q_i\}$.
By means of such a perturbation we can change
the configuration so that all six points
$p_1$, $p_2$, $p_3$, $q_1$, $q_2$, $q_3$
are distinct.
Then
\begin{align*}
[p_1p_2p_3]
&= [p_1p_2q_3][p_1p_3q_3][p_2p_3q_3]\quad&&\text{by \eqref{cyclic2}} \\
&= [p_1p_2q_3][p_1q_1q_3][p_2q_2q_3]\quad&&\text{by \eqref{cyclic4}} \\
&= [p_1p_2q_2][p_1q_3q_2][p_2q_3q_2] [p_1q_1q_3][p_2q_2q_3]\quad&&\text{by \eqref{cyclic2}} \\
&= [p_1p_2q_2][p_1q_2q_3][p_1q_1q_3][p_2q_2q_3]^2\quad&&\text{by \eqref{cyclic1}} \\
&= [p_1q_1q_2][p_1q_2q_3][p_1q_1q_3]\quad&&\text{by \eqref{cyclic0} and \eqref{cyclic4}} \\
&= [q_1q_2q_3]\quad&&\text{by \eqref{cyclic3}},
\end{align*}
and the lemma follows.
\end{proof}

Lemma \ref{cyclic-aux} and Lemma \ref{cyclic-aux2}
imply that the second requirement of Definition \ref{def-cyclic}
is satisfied for the map $(f_{p_1},f_{p_2},f_{p_3})$ from $D$ to $\R^3$.
Applying this to all triples $(p_i,p_j,p_k)$ where $1\le i<j<k\le n$
yields that this requirement is satisfied
for $f=(f_{p_1},\dots,f_{p_n})$,
and Proposition \ref{special-is-cyclic} follows.
\end{proof}

\subsection*{Proof of the inequality}
Fix an $\ep>0$.
By Proposition \ref{approx-santalo} applied to $\phi_0$,
there exists a positively cyclically ordered set of points
$p_1,\dots,p_n\in\pd D$ such that the map
$$
f^0=(f^0_1,\dots,f^0_n):D\to\R^n
$$
where
$$
 f^0_i(x) = d_{\phi_0}(p_i,x) \qquad x\in D,\ i=1,\dots,n,
$$
satisfies
$$
 I(f^0)>\vol_2^{ht}(D,\phi_0)-\ep .
$$
Define $f=(f_{p_1},\dots,f_{p_n}):D\to\R^n$
where the functions $f_{p_i}$ are defined by \eqref{specfun}.
Then Lemma \ref{same-on-boundary} implies that
$f|_{\pd D}=f^0|_{\pd D}$, hence
$I(f)=I(f^0)$ by Lemma \ref{l:I-by-boundary}.
On the other hand, Proposition \ref{special-is-cyclic}
implies that $f$ is cyclic (with respect to $\phi$)
and hence $\vol_2^{ht}(D,\phi)\ge I(f)$ by Proposition \ref{I<vol}.
Thus
$$
 \vol_2^{ht}(D,\phi)\ge I(f) = I(f^0) > \vol_2^{ht}(D,\phi_0)-\ep .
$$
Since $\ep$ is arbitrary,
it follows that $\vol_2^{ht}(D,\phi)\ge\vol_2^{ht}(D,\phi_0)$.
Thus we have proved the inequality part of Theorem \ref{main}.

\subsection*{The equality case}
Suppose that  $\vol_2^{ht}(D,\phi)=\vol_2^{ht}(D,\phi_0)$.
First we prove that $\phi$ is a metric with minimal geodesics.
Arguing by contradiction, suppose that a geodesic
$\gamma:[0,T]\to D\setminus\pd D$ of $\phi$
is not minimal, that is, $L_\phi(\gamma)>d_\phi(\gamma(0),\gamma(T))$.
Then a similar inequality holds for every geodesic of length $T$
with initial velocity sufficiently close to $v:=\dot\gamma(0)$.
Such geodesics cannot be parts of shortest paths, hence there
is a neighborhood of $v$ in $TD$ avoided by velocity
vectors of shortest paths.
A small perturbation of the metric in this neighborhood does
not change the boundary distances,
and one can choose a perturbation so that the resulting metric $\tilde\phi$
is such that $\tilde\phi\le\phi$ everywhere and $\tilde\phi(v)<\phi(v)$.
Then $\vol_2^{ht}(D,\tilde\phi)<\vol_2^{ht}(D,\phi)=\vol_2^{ht}(D,\phi_0)$,
contrary to the inequality part of the theorem
(applied to $\tilde\phi$ in place of $\phi$).
Thus $\phi$ is a metric with minimal geodesics.

Now suppose that the boundary distance functions of $\phi_0$ and $\phi$
differ. It is easy to see that then there are points $p,q\in\pd D$
such that $d_\phi(p,q)>d_{\phi_0}(p,q)$ and
a shortest path from $p$ to $q$ is a geodesic
which hits the boundary transversally.
Then one gets a contradiction
similarly to the above argument,
namely slightly shrinking the metric in a neighborhood
of a velocity vector of this geodesic.

Thus if $\vol_2^{ht}(D,\phi)=\vol_2^{ht}(D,\phi_0)$,
then $\phi$ is a metric with minimal geodesics
and has the same boundary distance function as~$\phi_0$.
Conversely, these two properties imply that the areas
are equal. To see this, just interchange $\phi$
and $\phi_0$ in the inequality part of the theorem.

This finishes the proof of Theorem \ref{main}.

\bibliographystyle{plain}

\end{document}